\documentclass{article}
\usepackage{amsmath,amsthm,amssymb,amsfonts}
\usepackage{verbatim}
\usepackage{graphicx}
\usepackage{stmaryrd} %for \leftrightarroweq
\usepackage{hyperref}

\usepackage[colorinlistoftodos]{todonotes}

\theoremstyle{plain}
\newtheorem{thm}{Theorem}
\newtheorem{lem}[thm]{Lemma}
\newtheorem{prop}[thm]{Proposition}
\newtheorem{cor}[thm]{Corollary}

\newtheorem{remark}[thm]{Remark}

\theoremstyle{definition}
\newtheorem{definition}[thm]{Definition}
\newtheorem{exl}[thm]{Example}

\numberwithin{thm}{section}

\newcommand{\adj}{\leftrightarrow}
\newcommand{\adjeq}{\leftrightarroweq}

\DeclareMathOperator{\id}{id}
\def\Z{{\mathbb Z}}
\def\N{{\mathbb N}}
\def\R{{\mathbb R}}

\begin{document}
%\begin{article}
%\begin{opening}
\title{Hyperspaces and Function Graphs in Digital Topology}
\author{Laurence Boxer
\thanks{
    Department of Computer and Information Sciences,
    Niagara University,
    Niagara University, NY 14109, USA;
    and Department of Computer Science and Engineering,
    State University of New York at Buffalo.
    email: boxer@niagara.edu
}
}

\date{ }
\maketitle{}

\begin{abstract}
We adapt the study of hyperspaces and 
function spaces from classical topology
to digital topology. We define digital hyperspaces and digital
function graphs, and
study some of their relationships and graphical properties.


\end{abstract}

%\classification{54C60}
%\end{opening}

\section{Introduction}
Classical topology has a large literature devoted to the
study of hyperspaces, in which a topology is induced on
some set of subsets of a given topological space. By the time
of the publication of~\cite{Nadler}, hundreds of papers had
been published on hyperspaces, and many more 
have appeared subsequently. Typically, the topology of a
hyperspace is induced by using the Hausdorff metric, which
essentially measures how two objects approximate each
other with respect to position. The Hausdorff metric 
can be computed efficiently~\cite{Shon89,BxMi}
and has been used by some students of digital image
processing as a crude measure of whether two images might
represent the same real-world object. Other metrics have
been developed in order to compare objects with respect to
topological or geometric 
properties~\cite{Borsuk54, Borsuk77, Cerin79, Cerin80, Boxer83, Boxer93}. Variations on the Hausdorff metric were introduced
in~\cite{CR96,Bx97,Vergili20,BxBeyond}

Classical topology also has a large literature on 
{\em function spaces}, in which the set of functions
\[ Y^X = \{f: X \to Y ~|~ f \mbox{ is continuous} \}
\]
between topological spaces, or some interesting subset of $Y^X$,
is considered as a topological
space whose topology is determined from those of $X$ and $Y$; see,
e.g., \cite{Borsuk67a,Wagner71,Boxer80,Boxer81,Boxer83b}.

In the current paper, we develop notions of hyperspaces and
function graphs (the latter, an analog of function spaces) for
digital topology. The paper is organized as follows.
\begin{itemize}
    \item Section~\ref{prelimSec} reviews basics of digital topology.
    \item In section~\ref{hyperspaceSec}, we introduce the adjacency
that we use to form a hyperspace of digital images.
\item Section~\ref{cardSec} has elementary observations on the
cardinalities of digital hyperspaces.
\item In section~\ref{mapSec} we discuss certain digitally
      continuous functions on hyperspaces. In 
      section~\ref{functGraphSubsec}, we introduce the concept
      of a function graph as a digital analog of a function space.
      Classical topology studies relations between hyperspaces and
      function spaces, e.g., \cite{Boxer80,Boxer81,Boxer83b}; 
      in section~\ref{functGraphSubsec} and later in the paper,
      we study relations between digital hyperspaces and 
      function graphs.
\item In section~\ref{connectSec}
we study connectedness properties of digital hyperspaces.
\item In section~\ref{multivaluedSec} we consider various
      notions of continuous multivalued functions in
      digital topology and their relations with
      digital hyperspaces.
\item In section~\ref{cycleSec} we obtain results concerning cycles and Girth in digital hyperspaces. 
\item In sections~\ref{dominateSec} and~\ref{diamSec}, we study, 
respectively, dominating sets and diameters 
of digital hyperspaces. 
\item We give some concluding remarks
in section~\ref{concludeSec}.
\end{itemize}

\section{Preliminaries}
\label{prelimSec}
Much of this section is quoted or paraphrased 
from~\cite{Boxer99,BxBeyond}.

We use $\N$ to indicate the set of natural numbers, $\Z$ 
for the set of integers, and $\R$ for the set of 
real numbers. We use~$\#X$ for the number of points in
a set~$X$.

\subsection{Adjacencies}
A digital image is a graph $(X,\kappa)$, where $X$ is a 
nonempty subset of $\Z^n$ for
some positive integer~$n$, and $\kappa$ is an 
adjacency relation for the points
of~$X$. The $c_u$-adjacencies are commonly used.
Let $x,y \in \Z^n$, $x \neq y$, where we consider these points as $n$-tuples of integers:
\[ x=(x_1,\ldots, x_n),~~~y=(y_1,\ldots,y_n).
\]
Let $u \in \N$,
$1 \leq u \leq n$. We say $x$ and $y$ are 
{\em $c_u$-adjacent} if
\begin{itemize}
\item There are at most $u$ indices $i$ for which 
      $|x_i - y_i| = 1$.
\item For all indices $j$ such that $|x_j - y_j| \neq 1$ we
      have $x_j=y_j$.
\end{itemize}
Often, a $c_u$-adjacency is denoted by the number of points
adjacent to a given point in $\Z^n$ using this adjacency.
E.g.,
\begin{itemize}
\item In $\Z^1$, $c_1$-adjacency is 2-adjacency.
\item In $\Z^2$, $c_1$-adjacency is 4-adjacency and
      $c_2$-adjacency is 8-adjacency.
\item In $\Z^3$, $c_1$-adjacency is 6-adjacency,
      $c_2$-adjacency is 18-adjacency, and $c_3$-adjacency
      is 26-adjacency.
%\item In $\Z^n$, $c_1$-adjacency is $2n$-adjacency and $c_n$-adjacency is $(3^n - 1)$-adjacency.
\end{itemize}

\begin{comment}
For $\kappa$-adjacent $x,y$, we write $x \adj_{\kappa} y$ or $x \adj y$ when $\kappa$ is understood. We write
$x \adjeq_{\kappa} y$ or $x \adjeq y$ to mean that either $x \adj_{\kappa} y$ or $x = y$.

Given subsets $A,B\subset X$, we say that $A$ and $B$ are \emph{adjacent} if
there exist points $a\in A$ and $b\in B$ such that
$a \adjeq b$. Thus sets with nonempty intersection are automatically adjacent, while disjoint sets may or may not be adjacent.
\end{comment}

We write $x \adj_{\kappa} x'$, or $x \adj x'$ when $\kappa$ is understood, to indicate
that $x$ and $x'$ are $\kappa$-adjacent. Similarly, we
write $x \adjeq_{\kappa} x'$, or $x \adjeq x'$ when $\kappa$ is understood, to indicate
that $x$ and $x'$ are $\kappa$-adjacent or equal.

A sequence $P=\{y_i\}_{i=0}^m$ in a digital image $(X,\kappa)$ is
a {\em $\kappa$-path from $a \in X$ to $b \in X$} if
$a=y_0$, $b=y_m$, and $y_i \adjeq_{\kappa} y_{i+1}$ 
for $0 \leq i < m$.

$Y \subset X$ is
{\em $\kappa$-connected}~\cite{Rosenfeld},
or {\em connected} when $\kappa$
is understood, if for every pair of points $a,b \in Y$ there
exists a $\kappa$-path in $Y$ from $a$ to $b$.

Let $N(X,x, \kappa)$ be the set
\[ N(X,x, \kappa) = \{y \in X ~|~ x \adj_{\kappa} y \}.
\]

\subsection{Digitally continuous functions}
In a metric space, the continuity of $f: X \to Y$
is defined to preserve the intuition that if
$x_0$ and $x_1$ are sufficiently close, then
$f(x_0)$ and $f(x_1)$ are close; i.e., ``closeness,"
and therefore connectivity, are preserved by a continuous 
function. Digital continuity is defined
to preserve connectedness, as at
Definition~\ref{continuous} below. By
using adjacency as our standard of ``closeness," we
get Theorem~\ref{continuityPreserveAdj} below.

\begin{definition}
\label{continuous}
{\rm ~\cite{Boxer99} (generalizing a definition of~\cite{Rosenfeld})}
Let $(X,\kappa)$ and $(Y,\lambda)$ be digital images.
A function $f: X \rightarrow Y$ is 
$(\kappa,\lambda)$-continuous if for
every $\kappa$-connected $A \subset X$ we have that
$f(A)$ is a $\lambda$-connected subset of $Y$.
\end{definition}

If $Y \subset X$, we use the abbreviation
$\kappa$-continuous for $(\kappa,\kappa)$-continuous.

When the adjacency relations are understood, we will simply say that $f$ is \emph{continuous}. Continuity can be expressed in terms of adjacency of points:

\begin{thm}
{\rm ~\cite{Rosenfeld,Boxer99}}
\label{continuityPreserveAdj}
A function $f:X\to Y$ is continuous if and only if $x \adj x'$ in $X$ implies 
$f(x) \adjeq f(x')$.
\end{thm}

See also~\cite{Chen94,Chen04}, where similar notions are referred to as {\em immersions}, {\em gradually varied operators},
and {\em gradually varied mappings}.

\begin{prop}
{\rm \cite{Boxer99}}
\label{compoPreservesContin}
If $f: (X,\kappa) \to (Y,\lambda)$ and
$g: (Y,\lambda) \to (W,\mu)$ are continuous maps between
digital images, then $g \circ f: X \to W$ is
$(\kappa,\mu)$-continuous.
\end{prop}

\begin{remark}
\label{pathRem}
Notice $P$ is a $\kappa$-path if and only if there
is a $(c_1, \kappa)$-continuous function
$p: [0,n]_{\Z} \to X$ such that $p([0,n]_{\Z})=P$.
It is therefore common to call such a
function a {\em $\kappa$-path}.
\end{remark}

To express the idea of following one path and then another,
the {\em product} or {\em concatenation} of paths is
defined as follows.

\begin{definition}
{\rm \cite{Khalimsky}}
Let $p_1: [0,m]_{\Z} \to X$ and $p_2: [0,n]_{\Z} \to X$
be $\kappa$-paths such that $p_1(m) = p_2(0)$.
The {\em product} or {\em concatenation} of these paths is
the function $p_1 \cdot p_2: [0,m+n]_{\Z} \to X$ given by
\[ (p_1 \cdot p_2)(t) = \left \{ \begin{array}{ll}
    p_1(t) & \mbox{if } 0 \le t \le m;  \\
    p_2(t-m) & \mbox{if } m \le t \le m+n.
\end{array}
\right .
\]
\end{definition}

\begin{lem}
{\rm \cite{Boxer94}}
The concatenation of paths is associative, i.e.,
\[ (p_1 \cdot p_2) \cdot p_3 = p_1 \cdot (p_2 \cdot p_3).
\]
\end{lem}

Let $Y \subset X$. A $\kappa$-continuous function $r: X \to Y$
is a {\em retraction}, and $Y$ is a {\em $\kappa$-retract of $X$},
if $r|_Y = \id_Y$.

A homotopy between continuous functions may be thought of as
a continuous deformation of one of the functions into the 
other over a finite time period.

\begin{definition}{\rm (\cite{Boxer99}; see also \cite{Khalimsky})}
\label{htpy-2nd-def}
Let $X$ and $Y$ be digital images.
Let $f,g: X \rightarrow Y$ be $(\kappa,\lambda)$-continuous functions.
Suppose there is a positive integer $m$ and a function
$F: X \times [0,m]_{{\Z}} \rightarrow Y$
such that
\begin{itemize}
\item for all $x \in X$, $F(x,0) = f(x)$ and $F(x,m) = g(x)$;
\item for all $x \in X$, the induced function
      $F_x: [0,m]_{{\Z}} \rightarrow Y$ defined by
          \[ F_x(t) ~=~ F(x,t) \mbox{ for all } t \in [0,m]_{{\Z}} \]
          is $(2,\lambda)-$continuous. That is, $F_x(t)$ is a path in $Y$.
\item for all $t \in [0,m]_{{\Z}}$, the induced function
         $F_t: X \rightarrow Y$ defined by
          \[ F_t(x) ~=~ F(x,t) \mbox{ for all } x \in  X \]
          is $(\kappa,\lambda)-$continuous.
\end{itemize}
Then $F$ is a {\rm digital $(\kappa,\lambda)-$homotopy between}
$f$ and $g$, and $f$ and $g$ are {\rm digitally 
$(\kappa,\lambda)-$homotopic in} $Y$, denoted
$f \sim_{\kappa,\lambda} g$.

If for some $x_0 \in X$ and $y_0 \in Y$
we have $F(x_0,t)=F(x_0,0) = y_0 \in Y$ for all
$t \in [0,m]_{{\Z}}$, we say $F$ {\rm holds $x$ fixed},
$F$ is a {\rm pointed homotopy}, and $x_0$ and $y_0$ are
{\em basepoints} of the homotopy.
$\Box$
\end{definition}

A different notion of digital homotopy has been introduced
by~\cite{LuptonEtal,Staecker}. The latter paper calls this
{\em strong homotopy}. It is defined as follows.

\begin{definition}
Let $X$ and $Y$ be digital images.
Let $f,g: X \rightarrow Y$ be $(\kappa,\lambda)$-continuous functions.
Suppose there is a positive integer $m$ and a function
$F: X \times [0,m]_{{\Z}} \rightarrow Y$
such that
\begin{itemize}
\item for all $x \in X$, $F(x,0) = f(x)$ and $F(x,m) = g(x)$; and
\item if $x \adjeq_{\kappa} y$ in $X$ and $t_0 \adjeq_{c_1} t_1$
      in $[0,m]_{{\Z}}$, then $F(x,t_0) \adjeq_{\lambda} F(y,t_1)$
      in $Y$.
\end{itemize}
Then $F$ is a {\em strong homotopy} between $f$ and $g$, and
$f$ and $g$ are strongly $(\kappa,\lambda)$-homotopic in $Y$.

If for some $x_0 \in X$ and $y_0 \in Y$
we have $F(x_0,t)=F(x_0,0) = y_0 \in Y$ for all
$t \in [0,m]_{{\Z}}$, we say $F$ {\rm holds $x_0$ fixed},
$F$ is a {\rm strong pointed homotopy}, and $x_0$ and
$y_0$ are {\em basepoints} of the homotopy.
\end{definition}

If there is a (strong) (pointed) $(\kappa,\kappa)$-homotopy 
$F: X \times [0,m]_{\Z} \to X$ between
the identity function $1_X$ and a constant function, we say 
$F$ is a (digital) {\em (strong) (pointed) $\kappa$-contraction} 
and $X$ is {\em (strongly) (pointed) $\kappa$-contractible}.

If there are continuous $f: (X,\kappa) \to (Y,\lambda)$ and
$g: (Y,\lambda) \to (X,\kappa)$ such that $g \circ f$ is
(strongly) (pointed) homotopic to $\id_X$ and $f \circ g$ is
(strongly) (pointed) homotopic to $\id_Y$, then
$(X,\kappa)$ and $(Y,\lambda)$ are
{\em (strongly) (pointed) homotopy equivalent} or 
{\em have the same (strong) (pointed) homotopy type}.

If $r: X \to X$ is a $\kappa$-retraction of $X$ to $Y \subset X$
that is (strongly) homotopic to $\id_X$, then $r$ is 
a ({\em strong} in the sense of digital homotopy) 
{\em deformation retraction}. If
a ({\em strong} in the sense of digital homotopy)
deformation retraction of $X$ to $Y \subset X$ holds fixed
every point of $Y$, then $r$ is a {\em strong}
(in the sense of deformation theory) 
({\em strong} in the sense of digital homotopy) {\em deformation retraction}.

If $f: (X, \kappa) \to (Y,\lambda)$ is a continuous bijection
such that $f^{-1}: (Y,\lambda) \to (X, \kappa)$ is continuous,
then $f$ is an {\em isomorphism} (called {\em homeomorphism}
in~\cite{Boxer94}) and $(X, \kappa)$ and $(Y,\lambda)$ are
{\em isomorphic}.

\section{Hyperspaces}
\label{hyperspaceSec}
The book~\cite{Nadler} is a good source for much of the material
discussed in this section that is taken from classical topology.

In classical topology, given a topological space $X$, we denote by
$2^X$ the set or {\em hyperspace} of nonempty compact subsets of $X$. 
If $X$ is a metric space with metric $d$, $2^X$ becomes a 
metric space with the Hausdorff metric, or some other
metric, based on $d$.

Given a digital image $(X,\kappa)$, we seek a somewhat parallel 
construction of a graph based on finite subsets of $X$. We let
\[ 2^X = \{Y ~ | ~ \emptyset \neq Y \subset X, ~\#Y < \infty\}.
\]
We define the $\kappa'$ adjacency for $2^X$ as follows.

\begin{definition}
\label{hyperAdjDef}
Let $\{A,B\} \subset 2^X$, $A \neq B$. Then $A \adj_{\kappa'} B$
if and only if given $a \in A$ and $b \in B$, there exist
$a_0 \in A$ and $b_0 \in B$
such that $a \adjeq_{\kappa} b_0$ and $b \adjeq_{\kappa} a_0$.
\end{definition}
The pair $(2^X,\kappa')$ is a graph or 
{\em tolerance space}~\cite{ZeeB}, the 
{\em hyperspace of $(X,\kappa)$}.
Note we do not call this hyperspace a digital image, since
$2^X$ is not a subset of $\Z^n$. However, since digital
topology's notions of continuous functions are defined in terms
of graph adjacency, or, alternately, graph connectedness, they
are naturally applied to this construction.

In classical topology, it is common to denote by $C(X)$ the subset of $2^X$
consisting of connected members of $2^X$. Since the notation
$C(X,\kappa)$ is established in the literature of digital topology as
the set of $\kappa$-continuous self maps on $X$, we use the notation
\[ K(X, \kappa') = 
   (\{A \in 2^X ~ | ~ A \mbox{ is $\kappa$-connected}\}, \kappa').
\]
We will use the abbreviation $K(X)$ when $\kappa$ is understood.

\begin{exl}
$K([a,b]_{\Z},c_1')$ and 
          $(\{(x,y) \in \Z^2 ~|~ a \le x \le y \le b\}, c_2)$
   are isomorphic graphs.
\end{exl}

\begin{proof}
Let $X = K([a,b]_{\Z},c_1')$, 
$Y=\{(x,y) \in \Z^2 ~|~ a \le x \le y \le b\}$. Consider the
function $F: X \to Y$ given by $F([m,n]_{\Z}) = (m,n)$. It is
elementary to show that $F$ is a $(c_1',c_2)$-isomorphism.
\end{proof}

\section{Cardinality}
\label{cardSec}
\begin{remark}
Let $(X,\kappa)$ be a digital image such that $\#X = n$.
Then $\#2^X = 2^n - 1$. This is because for each
$x \in X$ and $A \in 2^X$, either $x \in A$ or $x \not \in A$.
This yields $2^n$ possible combinations of pixels, but we exclude the
empty set.
\end{remark}

However, the following example shows that $K(X)$ may be considerably
smaller than $2^X$.

\begin{exl}
\label{cardOfInterval}
$\#K([1,n]_{\Z},c_1') = n(n+1)/2$.
\end{exl}

\begin{proof}
For $i \in [1,n]_{\Z}$, the members of 
$K([1,n]_{\Z},c_1')$ that
have $i$ as their largest member are those of 
$\{[j,i]_{\Z}\}_{j=1}^i\}$. Since there are $i$ digital
intervals with largest member $i$ in $K([1,n]_{\Z},c_1')$,
\[ \#K([1,n]_{\Z},c_1') = \sum_{i=1}^n i = n(n+1)/2.
\]
\end{proof}

\section{Maps on digital hyperspaces}
\label{mapSec}
In this section, we study maps induced on hyperspaces by
continuous maps between digital images.

\subsection{Induced maps}
Given a continuous map $f: (X,\kappa) \to (Y,\lambda)$,
we show below that $f$ induces a 
$(\kappa',\lambda')$-continuous map
$f_*: 2^X \to 2^Y$ such that
$f_*|_{K(X)}: K(X) \to K(Y)$ is also
$(\kappa',\lambda')$-continuous. In the following,
we will use the notation $f_*$ to abbreviate $f_*|_{K(X)}$.

\begin{thm}
\label{cont-thm}
Let $(X,\kappa)$ and $(Y,\lambda)$ be digital images and let
$f: X \to Y$. Then $f$ is $(\kappa,\lambda)$-continuous if and only if
the induced functions $f_*: 2^X \to 2^Y$ and
$f_*: K(X) \to K(Y)$
defined by $f_*(A) = f(A)$ are
$(\kappa',\lambda')$-continuous.
\end{thm}

\begin{proof}
Since a digitally continuous function preserves 
adjacency and connectivity,
the same argument works for both of the induced functions.

Suppose $f$ is continuous. Let $A,B \in 2^X$ such that
$A \adj_{\kappa'} B$. Let $x \in A$.
There exists $y \in B$ such that $x \adjeq_{\kappa} y$. By
the continuity of $f$ we have 
$f(x) \adjeq_{\lambda} f(y)$. Similarly, for
$b \in B$, there exists $a \in A$ such that
$a \adjeq_{\kappa} b$
and $f(a) \adjeq_{\lambda} f(b)$. Therefore,
$f_*(A) = f(A) \adjeq_{\lambda'} f(B) = f_*(B)$. Thus $f_*$ is
$(\kappa', \lambda')$-continuous.

Suppose $f_*$ is $(\kappa', \lambda')$-continuous. Let $x,y \in X$ such that
$x \adj_{\kappa} y$. Then $\{x\} \adj_{\kappa'} \{y\}$, so
\[ \{f(x)\} = f_*(\{x\}) \adjeq_{\lambda'} f_*(\{y\}) = \{f(y)\}.
\]
Therefore, $f(x) \adjeq_{\lambda} f(y)$. Thus, $f$ is $(\kappa,\lambda)$-continuous.
\end{proof}

We have the following as an immediate consequence of Theorem~\ref{cont-thm}.

\begin{cor}
Let $(X,\kappa)$ and $(Y,\lambda)$ be digital images and let
$f: X \to Y$. Then the following are equivalent.
\begin{itemize}
    \item $f$ is a $(\kappa,\lambda)$-isomorphism;
    \item the induced function $f_*: 2^X \to 2^Y$ is a
$(\kappa',\lambda')$-isomorphism; and
    \item the induced function $f_*: K(X) \to K(Y)$ is a
$(\kappa',\lambda')$-isomorphism
\end{itemize} 
\end{cor}

\begin{prop}
\label{compositionProp}
Given continuous functions $f: (X,\kappa) \to (Y,\lambda)$
and $g: (Y,\lambda) \to (W,\mu)$, we have
$(g \circ f)_* = g_* \circ f_*$.
\end{prop}

\begin{proof}
The assertion follows from the observation that $A \in 2^X$ implies
\[ (g \circ f)_*(A) = (g \circ f)(A) = g(f(A)) = 
   g_*(f_*(A)) = (g_* \circ f_*)(A).
\]
\end{proof}

The following is elementary.

\begin{lem}
\label{idLemma}
Let $(X,\kappa)$ be a digital image. Then
$(\id_X)_* = \id_{(2^X,\kappa')}$ and
$(\id_X)_* = \id_{(K(X),\kappa')}$.
\end{lem}

\begin{thm}
The hyperspace construction yields covariant functors $F, F'$ from 
the category of digital images and continuous functions to
the category of graphs and continuous functions (respectively, to
the category of connected graphs and continuous functions), in which
$F(X,\kappa) = (2^X, \kappa')$, $F'(X,\kappa) = K(X,\kappa')$ and for
$f: (X,\kappa) \to (Y,\lambda)$ we have $F(f)=f_*$, 
$F'(f)=f_*$.
\end{thm}

\begin{proof}
This follows from Proposition~\ref{compositionProp} and
Lemma~\ref{idLemma}.
\end{proof}

Not every continuous function on the hyperspace 
of a digital image is
induced by a continuous map between digital images,
as shown by the following.

\begin{exl}
Let $X = [0,1]_{\Z}$.
Let $F: K(X, c_1') \to K(X, c_1')$ be the function
given by $F(A) = X$ for all $A \in K(X)$. 
$F$ is constant, hence continuous, 
and is not induced by any $f: X \to X$ since for each such
function, e.g., $f(0) \in X = \{\,0,1\,\} = f_*(0)$, hence
$f_*(\{\, 0 \,\}) \neq F(0)$.
\end{exl}

\subsection{Retraction and homotopy}

\begin{thm}
\label{retractThm}
Let $(X,\kappa)$ and $(Y,\kappa)$ be digital images and let
$r: X \to Y$ be a $\kappa$-retraction. Then the induced
maps $r_*: 2^X \to 2^Y$ and $r_*: K(X) \to K(Y)$ are $\kappa'$-retractions.
\end{thm}

\begin{proof}
It follows from Theorem~\ref{cont-thm} that each version
of $r_*$ is $\kappa'$-continuous.
It is clear that $r_*|_{2^Y} = \id_{2^Y}$, 
$r_*(2^X) = 2^Y$, $r_*(K(X)) = K(Y)$,
$r_*|_{K(Y)} = \id_{K(Y)}$. The assertion follows.
\end{proof}

\begin{thm}
\label{htpyInduction}
Let $f$ and $g$ be (strongly) (pointed) homotopic 
maps from $(X,\kappa)$ to $(Y, \lambda)$.
Then $f_*$ and $g_*$ are (strongly) (pointed) homotopic 
maps from $(2^X,\kappa')$ to $(2^Y, \lambda')$, and
from $K(X)$ to $K(Y)$. In the case
of (strongly) pointed homotopy, if $x_0 \in X$ is held
fixed by the (strong) pointed homotopy from $f$ to $g$,
then $\{x_0\}$ is held fixed by the (strong)
pointed homotopy from $f_*$ to $g_*$.
\end{thm}

\begin{proof}
We give a proof for homotopy using $2^X$ and $2^Y$.
The proofs for strong or pointed homotopies
and for $K(X)$ and $K(Y)$ are similar.

By hypothesis, there is a function
$H: X \times [0,n]_{\Z} \to Y$ for some $n \in \N$ such that
\begin{itemize}
    \item $H(x,0) = f(x)$ and $H(x,n) = g(x)$ for all $x \in X$.
    \item For all $x \in X$, the induced function $H_x: [0,n]_{\Z} \to Y$ defined by
          $H_x(t) = H(x,t)$ is $(c_1,\lambda)$-continuous.
    \item For all $t \in [0,n]_{\Z}$, the induced function
          $H_t: X \to Y$ defined by $H_t(x) = H(x,t)$ is $(\kappa, \lambda)$-continuous.
\end{itemize}
Let $H_*: 2^X \times [0,n]_{\Z} \to 2^Y$ be the function
$H_*(A,t) = H_t(A)$.
\begin{itemize}
    \item We have
    \[H_*(A,0) = H_0(A) = f(A) = f_*(A).
    \]
    Similarly, $H_*(A,n) = g_*(A)$.
    \item For all $A \in 2^X$, consider the induced function
          $H_{*,A}: [0,n]_{\Z} \to 2^Y$ defined by
          \[H_{*,A}(t) = H(A,t) = 
          \bigcup_{x \in A} \{H_x(t)\}.
          \]
          Since $0 \le t < n$
          implies $H_x(t) \adjeq_{\lambda} H_x(t+1)$, it follows that
          $H_{*,A}(t) \adjeq_{\lambda'} H_{*,A}(t+1)$. Thus the induced function
          $H_{*,A}$ is $(c_1, \lambda')$-continuous.
    \item For all $t \in [0,n]_{\Z}$, consider the induced function
          $H_{*,t}: 2^X \to 2^Y$ given by
          \[ H_{*,t}(A) = H(A,t) = 
          \bigcup_{x \in A} \{H_t(x)\}.
          \]
          If $A \adj_{\kappa'} B$ then for each $a \in A$ there exists $b \in B$
          such that $a \adjeq_{\kappa} b$, and for each $\beta \in B$ there exists
          $\alpha \in A$ such that $\alpha \adjeq_{\kappa} \beta$. Therefore,
          $H_t(a) \adjeq_{\lambda} H_t(b)$ and 
          $H_t(\alpha) \adjeq_{\lambda} H_t(\beta)$. It follows that
          $H_{*,t}(A) \adjeq_{\lambda'} H_{*,t}(B)$.
          Thus $H_{*,t}$ is $(\kappa', \lambda')$-continuous.
\end{itemize}
The above shows that $H_*$ is a homotopy from $f_*$ to $g_*$.
\end{proof}

\begin{thm}
\label{htpyThm}
Let $(X,\kappa)$ and $(Y,\lambda)$ be digital images.
\begin{enumerate}
    \item If $(X,\kappa)$ and $(Y,\lambda)$ have the same 
          (pointed) homotopy type,
then $(2^X,\kappa')$ and $(2^Y,\lambda')$ have the 
same (pointed) homotopy type; as do 
$K(X,\kappa')$ and $K(Y, \lambda')$.
    \item If $(X,\kappa)$ and $(Y,\lambda)$ have the same strong
          (pointed) homotopy type, then $(2^X,\kappa')$ 
          and $(2^Y,\lambda')$ have 
          the same strong (pointed) homotopy type; as do 
$K(X,\kappa')$ and $K(Y, \lambda')$.
    \item Let $(X,\kappa)$ be (pointed) contractible (respectively,
          (pointed) strongly contractible). Then $(2^X,\kappa')$ is contractible 
          (respectively, (pointed) strongly contractible); as is $K(X,\kappa')$.
\end{enumerate}

\end{thm}

\begin{proof}
We give proofs for the full hyperspaces
$2^X$ and $2^Y$; the proofs for $K(X,\kappa')$
and $K(Y,\lambda')$ are similar.

\begin{enumerate}
    \item By hypothesis, there are continuous (pointed)
          maps $f: (X,\kappa) \to (Y,\lambda)$ and 
$g: (Y,\lambda) \to (X,\kappa)$ (with basepoints $x_0 \in X$,
$y_0 \in Y$) such that
$g \circ f \sim_{(\kappa, \kappa)} \id_X$ 
(holding $x_0$ fixed) and
$f \circ g \sim_{(\lambda, \lambda)} \id_Y$ 
(holding $y_0$ fixed).
By Proposition~\ref{compositionProp} and 
Theorem~\ref{htpyInduction}, 
\[g_* \circ f_* = (g \circ f)_* \sim_{(\kappa', \kappa')} (\id_X)_* = \id_{2^X} \mbox{ (holding $\{x_0\}$ fixed)}
\]
and
\[f_* \circ g_* = (f \circ g)_* \sim_{(\lambda', \lambda')} (\id_Y)_*
  = \id_{2^Y} \mbox{ (holding $\{y_0\}$ fixed)}.
\]
Therefore, $(2^X,\kappa')$ and $(2^Y,\lambda')$ have the
same homotopy type.
\item The proof for strong homotopy type is similar.
\item Since (pointed) contractible (respectively, 
(pointed) strongly contractible) means having
the same (pointed) homotopy type (respectively, (pointed) 
strong homotopy type) as
a digital image of a single point, the assertions concerning
(pointed) contractibility (respectively, (pointed) strong
contractibility) follow from the above.
\end{enumerate}
\end{proof}

\begin{thm}
Let $H: (X,\kappa) \times [0,n]_{\Z} \to (X,\kappa)$ be a 
(strong, in the sense of strong homotopy)
(strong, in the sense of retraction theory)
deformation retraction of $X$ to a subset $Y$,
i.e., a (strong) homotopy between the induced maps 
$H_0, H_n: X \to X$ such 
that $H_0 = \id_X$ and $H_n$ is a retraction
(that holds fixed every point of $Y$). Then 
$H_*: (2^X,\kappa') \times [0,n]_{\Z} \to (2^X,\kappa')$ 
(respectively, $H_*: (K(X),\kappa') \times [0,n]_{\Z} \to (K(X),\kappa')$)
is a (strong, in the sense of strong homotopy)
(strong, in the sense of retraction theory)
deformation retraction of $(2^X,\kappa')$ to $(2^Y,\kappa')$
(respectively, of $(K(X),\kappa')$ to $(K(X),\kappa')$).
\end{thm}

\begin{proof}
These assertions follow from Theorems~\ref{htpyThm} and~\ref{retractThm}.
\end{proof}

\subsection{Function graphs}
\label{functGraphSubsec}
In this section, we explore an analog of function spaces
for digital images.

\begin{definition}
Let $(X,\kappa)$ and $(Y,\lambda)$ be digital images.
Consider the set $(Y,\lambda)^{(X,\kappa)}$, or $Y^X$ when 
$\kappa$ and $\lambda$ can be assumed, defined by
\[ (Y,\lambda)^{(X,\kappa)} = \{f: X \to Y ~|~ f
    \mbox{ is $(\kappa,\lambda)$-continuous}\}.
\]
We say $f,g \in Y^X$ are {\em $\Phi(\kappa,\lambda)$-adjacent}, or
{\em $\Phi$-adjacent} when $\kappa$ and $\lambda$ can be assumed, if
for all $x \in X$ we have $f(x) \adjeq_{\lambda} g(x)$.
\end{definition}

A more restrictive adjacency for $Y^X$, which we 
denote as $\Psi(\kappa,\lambda)$,
is proposed in~\cite{LuptonEtal}. We have the following.

\begin{definition}
\label{PsiAdj}
{\rm ~\cite{LuptonEtal}}
Let $f,g \in Y^X$. Then $f \adj_{\Psi(\kappa,\lambda)} g$
if given $x_0 \adjeq_{\kappa} x_1$ in $X$,
$f(x_0) \adjeq_{\lambda} g(x_1)$ in $Y$.
\end{definition}

\begin{remark}
It is clear that $f \adj_{\Psi(\kappa,\lambda)} g$ implies
$f \adj_{\Phi(\kappa,\lambda)} g$. The converse is not
generally valid. For example, consider the functions
$f,g: [0,2]_{\Z} \to [0,2]_{\Z}$ given by
$f(x) = x$, $g(x) = \min\{2,x+1\}$. It is easily seen
that $f,g \in C([0,2]_{\Z},c_1)$ and
$f \adj_{\Phi(c_1,c_1)} g$. However, since
$0 \adj_{c_1} 1$ and $f(0)=0 \not \adj_{c_1} 2 = g(1)$,
$f$ and $g$ are not $\Psi(c_1,c_1)$-adjacent.
\end{remark}

We show below, at Example~\ref{cycleAndAdjs},
an important difference between $\Phi(c_1,c_1)$ 
and $\Psi(c_1,c_1)$.

\begin{lem}
\label{1StepLem}
Let $(X,\kappa)$ and $(Y,\lambda)$ be digital images.
Let $f,g \in Y^X$. Then $f$ and $g$ are
homotopic in one step if and only if
$f \adjeq_{\Phi} g$.
\end{lem}

\begin{proof}
Suppose $f$ and $g$ are homotopic in one step. Then there
exists $H: X \times [0,1]_{\Z} \to Y$
such that for all $x \in X$, $H(x,0)=f(x)$ and
$H(x,1)=g(x)$, and the induced function
$H_x: [0,1]_{\Z} \to Y$ given by $H_x(t)=H(x,t)$ is
$(c_1,\lambda)$ continuous. The latter implies
\[f(x) = H(x,0) \adjeq_{\lambda} H(x,1) = g(x)
\]
for all $x \in X$, so $f \adjeq_{\Phi} g$.

Suppose $f \adjeq_{\Phi} g$. Then one sees easily that
the function $H: X \times [0,1]_{\Z} \to Y$
defined by
\[ H(x,0) = f(x), ~~~ H(x,1) = g(x),
\]
is a homotopy in one step from $f$ to $g$.
\end{proof}

The following was suggested by an anonymous reviewer.

\begin{thm}
\label{homotopyInComponent}
Let $(X,\kappa)$ and $(Y,\lambda)$ be digital images.
Let $f,g \in Y^X$. Then $f$ and $g$ are
homotopic if and only if $f$ and $g$ belong to the 
same component of $(Y^X, \Phi)$.
\end{thm}

\begin{proof}
Since both homotopy between functions and being 
connected by a path are transitive relations, the 
assertion follows from Lemma~\ref{1StepLem}.
\end{proof}

Let $(S_n,\kappa)$ be any cyclic graph of $n > 4$ points,
with point set $S_n=\{x_i\}_{i=0}^{n-1}$ such that
$x_i \adj_{\kappa} x_j$ if and only if
$|i-j| \in \{1,n-1\}$. Let $r_j \in C(S_n,\kappa)$
be the rotation $r_j(x_i) = x_{(i+j) \mod n}$.
We have the following.

\begin{thm}
{\rm \cite{BxSt20}}
\label{onlyRotate}
If $f \in C(S_n, \kappa)$ such that $f$ and $\id_{S_n}$
are $\kappa$-homotopic, then $f=r_j$ for some $j$,
$0 \le j < n$.
\end{thm}

We do not get a similar outcome if we substitute
$\Psi$ for $\Phi$ in Theorem~\ref{homotopyInComponent},
as shown in the following.

\begin{exl}
\label{cycleAndAdjs}
Let $(S_n,\kappa)$ be any cyclic graph of $n > 4$ points.
Then all the rotations $r_j$ are homotopic. However,
no distinct $r_j$ and $r_k$ belong to the same
component of $(S_n^{S_n}, \Psi)$.
\end{exl}

\begin{proof}
Without loss of generality, $k = j + m < n$ for some $m$,
$0 < m < n-j$. Then
$H: S_n \times [0,m]_{\Z} \to S_n$, defined by
$H(x_i,t) = r_{j+t}(x_i)$, is a homotopy from
$r_j$ to $r_k$.

It follows from Theorem~\ref{onlyRotate} that every
induced map $H_t$ of $H$ for $t \in [0,m]_{\Z}$,
and in particular, $H_1$, is a rotation.
\begin{itemize}
    \item For $1 \le m<n-2$, $r_j(x_0)=x_j$ and 
          $r_k(x_1)=x_{(j+m+1) \mod n}$ are not $\kappa$-adjacent.
    \item For $m=n-2$, we cannot follow the pattern 
          used above, since
          \[ r_k(x_1) = r_{j+n-2}(x_1) = 
             r_{j-2 \mod n}(x_1) = x_{j-1 \mod n}
          \]
          is adjacent to $r_j(x_0)$. However,
          \[ r_k(x_0) = r_{j+n-2}(x_0) = 
             r_{j-2 \mod n}(x_0) = x_{j-2 \mod n}
          \]
          is neither adjacent nor equal to $r_j(x_0)$.
    \item For $m=n-1$ we must have $j=0$. Therefore,
          $r_j(x_1)=x_1$ and $r_k(x_0)=x_{n-1}$ are
          not $\kappa$-adjacent.
\end{itemize}
In every case, $r_j$ and $r_k$ are not $\Psi$-adjacent.
This completes the proof.
\end{proof}

\begin{thm}
Let $(X,\kappa)$ and $(Y,\lambda)$ be digital images.
Let $W$ be a $\lambda$-retract of $Y$. Then
$(W^X, \Phi(\kappa,\lambda))$ is a retract of
$(Y^X, \Phi(\kappa,\lambda))$.
\end{thm}

\begin{proof}
Let $r: Y \to W$ be a $\lambda$-retraction. Then
for every $(\kappa,\lambda)$-continuous $f: X \to Y$,
$r \circ f:X \to W$ is continuous by 
Proposition~\ref{compoPreservesContin}. Further,
if $f(X) \subset W$ then $r \circ f = f$. The assertion
follows.
\end{proof}

We present results that link the topics of 
hyperspaces and function graphs.

\begin{thm}
Let $f \adj_{\Phi(\kappa,\lambda)} g$ in $Y^X$. Then for
$A \in 2^X$, $f(A) \adjeq_{\lambda'} g(A)$.
\end{thm}

\begin{proof}
Let $y_f \in f(A)$. Let $x_f \in f^{-1}(y_f)$. Then
$y_f=f(x_f) \adjeq_{\lambda} g(x_f)$. Similarly, given
$y_g \in g(A)$, there exists $x_g \in g^{-1}(y_g)$ such that
$f(x_g) \adjeq_{\lambda} g(x_g) = y_g$. It follows that
$f(A) \adjeq_{\lambda'} g(A)$.
\end{proof}

\begin{thm}
\label{htpyPreserved}
Let $(W,\kappa)$, $(X,\lambda)$, and $(Y,\mu)$ be digital images.
Suppose $f,g \in Y^X$ are $(\lambda,\mu)$-continuous. If
$f$ and $g$ are
\begin{itemize}
    \item (strongly) $(\lambda,\mu)$-homotopic;
    \item (strongly) pointed $(\lambda,\mu)$-homotopic
          with $x_0$ held fixed,
\end{itemize}
then the induced maps 
$f_*, g_*: (X^W, \Phi(\kappa,\lambda)) \to (Y^W, \Phi(\kappa,\mu))$,
defined for all $F \in X^W$ by
\[ f_*(F) = f \circ F, ~~~~~ g_*(F) = g \circ F,
\]
are $(\Phi(\kappa,\lambda), \Phi(\kappa,\mu))$-continuous and, respectively, $f_*$ and $g_*$ are,
\begin{itemize}
    \item (strongly)
    $(\Phi(\kappa,\lambda),\Phi(\kappa,\mu))$-homotopic;
    \item (strongly) pointed
    $(\Phi(\kappa,\lambda),\Phi(\kappa,\mu))$-homotopic 
          with the constant function $\hat{x}_0$ 
          held fixed.
\end{itemize}
\end{thm}

\begin{proof}
Let $F, G \in X^W$ be $(\kappa,\lambda)$-continuous
with $F \adj_{\Phi(\kappa, \lambda)} G$ and
let $w \in W$. Then 
\[ F(w) \adjeq_{\lambda} G(w), \mbox{ so }
   f_*(F)(w) = f \circ F(w) \adjeq_{\mu} f \circ G(w) =
   f_*(G)(w),
\]
so $f_*(F) \adjeq_{\Phi(\kappa,\mu)} f_*(G)$, hence
$f_*$ is continuous. Similarly, $g_*$ is continuous.

We proceed with a proof for homotopic maps;
the other assertions are proven similarly. 

Let $H: X \times [0,n]_{\Z} \to Y$ be a 
$(\lambda,\mu)$-homotopy from $f$ to $g$. Let 
$H_*: (X^W, \Phi(\kappa,\lambda)) \times [0,n]_{\Z} \to (Y^W, \Phi(\kappa,\mu))$ be given by
$H_*(F,t)(x) = H(F(x),t)$.
We have the following.
\begin{itemize}
    \item $H_*(F,0)(x) = H(F(x),0) = f(F(x))=f_*(F)(x)$, so
          $H_*|_{t=0} = f_*$; and
          $H_*(F,n)(x) = H(F(x),n) = g(F(x))=g_*(F)(x)$, so
          $H_*|_{t=n} = g_*$.
    \item Given $F \in X^W$,
          the induced function $H_{*,F}: [0,n]_{\Z} \to Y^W$
          given by $H_{*,F}(t)(w)=H(F(w), t)$ satisfies, for
          $t_0 \adj_{c_1} t_1$ in $[0,n]_{\Z}$,
          \[H_{*,F}(t_0)(w) = H(F(w), t_0) \adjeq_{\mu} 
            H(F(w), t_1)(w) = H_{*,F}(t_1)(w),
          \]
          so $H_{*,F}$ is $(c_1,\Phi(\lambda,\mu))$-continuous.
    \item Given $t \in [0,n]_{\Z}$, the induced function 
          $H_{*,t}: X^W \to Y^W$ given by
          $H_{*,t}(F)(w)=H(F(w), t)$ satisfies, for
          $F_0 \adj_{\Phi(\kappa,\lambda)} F_1$ in $X^W$,
          \[ H_{*,t}(F_0)(w)= H(F_0(w),t) \adjeq_{\mu} 
             H(F_1(w),t) = H_{*,t}(F_1)(w).
          \]
          Therefore, $H_{*,t}$ is 
          $(\Phi(\kappa,\lambda), \Phi(\kappa,\mu))$-continuous.
\end{itemize}
Therefore, $H_*$ is a homotopy from $f_*$ to $g_*$.
\end{proof}

\begin{prop}
\label{compositionInduced}
Let $(V,\kappa)$, $(W,\lambda)$, $(X,\mu)$, $(Y,\nu)$ be digital
images. Let
$f: (W,\lambda) \to (X,\mu)$ and $g: (X,\mu) \to (Y,\nu)$
be continuous. Consider the induced maps
$f_*: W^V \to X^V$ and $g_*: X^W \to Y^W$. We have
$(g \circ f)_* = g_* \circ f_*: W^V \to Y^V$.
\end{prop}

\begin{proof}
Given $F: V \to W$, we have
\[ g_* \circ f_* (F) = g_* (f \circ F) = g \circ (f \circ F) =
   (g \circ f) \circ F = (g \circ f)_*(F).
\]
The assertion follows.
\end{proof}

\begin{cor}
\label{htpyTypeFunctGraph}
Let $(X,\kappa)$ and $(Y, \lambda)$ be digital images.
Suppose $(X,\kappa)$ and $(Y, \lambda)$ have
the same (strong) (pointed) homotopy type.
Then $(X^X,\Phi(\kappa,\kappa))$ and
$(Y^Y, \Phi(\lambda,\lambda))$ have
the same (strong) (pointed) homotopy type, respectively.
\end{cor}

\begin{proof}
We give a proof for ``same homotopy type"; the other assertions
are established similarly (in the pointed cases, if $x_0 \in X$
and $y_0 \in Y$ are the basepoints of the assumption,
then the constant maps $\hat{x_0} \in X^X$ and $\hat{y_0} \in Y^Y$
are the basepoints of the conclusion).

If $(X,\kappa)$ and $(Y, \lambda)$ have the same homotopy type,
then there are continuous functions $f: X \to Y$ and
$g: Y \to X$ such that $g \circ f \sim_{\kappa,\kappa} \id_X$
and $f \circ g \sim_{\lambda,\lambda} \id_Y$. By
Theorem~\ref{htpyPreserved} and Proposition~\ref{compositionInduced},
we have
\[ (g \circ f)_* \sim_{\Phi(\kappa,\kappa)} (\id_X)_* = \id_{X^X}
\]
and similarly, 
\[ (f \circ g)_* \sim_{\Phi(\lambda,\lambda)} (\id_Y)_* = \id_{Y^Y}.
\]
Thus, $(X^X, \Phi(\kappa,\kappa))$ and
$(Y^Y, \Phi(\lambda,\lambda))$
have the same homotopy type.
\end{proof}

\begin{cor}
Let $(X,\kappa)$ be a digital image. 
Suppose $(X,\kappa)$ is
\begin{itemize}
    \item (strongly) contractible;
    \item (strongly) pointed contractible with basepoint $x_0$.
\end{itemize}
Then, respectively, $(X^X, \Phi(\kappa,\kappa))$ is
\begin{itemize}
    \item (strongly) contractible;
    \item (strongly) pointed contractible with the constant function
          $\hat{x_0}$ as basepoint.
\end{itemize}
\end{cor}

\begin{proof}
We give a proof for ``contractible"; the other assertions follow
similarly.

Since ``contractible" means homotopy equivalent to a digital
image with a single point, the assertion follows from
Corollary~\ref{htpyTypeFunctGraph}.
\end{proof}

\section{Connectedness in digital hyperspaces}
\label{connectSec}
We have the following.

\begin{prop}
\label{amalgConn}
Let $(X,\kappa)$ be a digital image. Let $W$ be a nonempty
$\kappa'$-connected subset of $K(X)$. Then
$W' = \bigcup_{Y \in W} Y$ is a $\kappa$-connected subset of $X$.
\end{prop}

\begin{proof}
Let $x_0,x_1 \in W'$. There exist $Y_i \in W$ such that
$x_i \in Y_i \in K(X, \kappa')$. Since $W$ is $\kappa'$-connected, there
exists a $\kappa'$-path $\{W_i\}_{i=0}^n \subset W$ from $Y_0$ to $Y_1$,
i.e., $W_0 = Y_0$, $W_i \adj_{\kappa'} W_{i+1}$, and $W_n = Y_1$.

By Definition~\ref{hyperAdjDef}, there exist $p_i \in W_i$, 
$q_{i+1} \in W_{i+i}$ such that
$p_i \adjeq_{\kappa} q_{i+1}$. As each $W_i$ is $\kappa$-connected,
there exist $\kappa$-paths $P_0 \subset Y_0=W_0$ from $x_0$ to $p_0$;
$P_i \subset W_i$ from $q_i$ to $p_{i+1}$, $1 \le i < n$; and 
$P_n \subset W_n = Y_1$ from $p_n$ to $x_1$.

Then $\bigcup_{i=0}^n P_i$ is a $\kappa$-path in $W'$ from
$x_0$ to $x_1$. It follows that $W'$ is $\kappa$-connected.
\end{proof}

\begin{prop}
\label{disconnXimpliesdisconnKX}
Let $C_0$ and $C_1$ be distinct components of $(X,\kappa)$. Let 
\begin{equation}
    \label{endsInComponents}
    A_i \in K(C_i, \kappa') \mbox{ for } i \in \{0,1\}.
\end{equation}
Then $A_0$ and $A_1$ are points of distinct components of $K(X,\kappa')$.
\end{prop}

\begin{proof}
Were $A_0$ and $A_1$ in the same component of $K(X,\kappa')$, 
then there would exist a path $\{B_j\}_{j=0}^n \subset K(X,\kappa')$
such that $A_0 = B_0$, 
\begin{equation}
\label{B_jAdjacencies}
    B_j \adj_{\kappa'} B_{j+1} \mbox{ for } 1 \le j < n,
\end{equation} 
and $B_n = A_1$.
By~(\ref{endsInComponents}), there is a smallest $k \in \N$ such
that $0 \le k < n$, $B_k \subset C_0$, and $B_{k+1} \not \subset C_0$.
But by~(\ref{B_jAdjacencies}), $B_k \cup B_{k+1}$ is
$\kappa$-connected and therefore must be a
subset of $C_0$, contrary to our choice of $k$. It follows that
$A_0$ and $A_1$ are points of distinct components of $K(X,\kappa')$.
\end{proof}

\begin{prop}
\label{componentYieldsConnected}
Let $(X,\kappa)$ be a finite connected digital image. Then
$K(X,\kappa')$ is connected.
\end{prop}

\begin{proof}
Let $A \in K(X,\kappa')$. We show there is a path in
$K(X,\kappa')$ from $A$ to $X$. If $A=X$, we are done. Otherwise,
since $X$ is connected, there are sequences
$\{x_i\}_{i=1}^m \subset X \setminus A$ and $\{A_j\}_{j=0}^m$
such that $A = A_0$, $A_{j+1} = A_j \cup \{x_{j+1}\}$, $A_{j+1}$
is connected, and $A_m = X$. Therefore, $A_j \adj_{\kappa'} A_{j+1}$.
Thus $\{A_j\}_{j=0}^m$ is a $\kappa'$-path in $K(X,\kappa')$ from
$A$ to $X$.

Since $A$ was arbitrarily chosen, it follows that $K(X,\kappa')$ is
connected.
\end{proof}

\begin{prop}
\label{DconnImpliesKDconn}
Let $D$ be a component of $(X,\kappa)$. Then
$K(D,\kappa')$ is a component of $K(X,\kappa')$.
\end{prop}

\begin{proof}
By Proposition~\ref{componentYieldsConnected}, 
$K(D,\kappa')$ is connected.
The conclusion follows from
Proposition~\ref{disconnXimpliesdisconnKX}.
\end{proof}

\begin{thm}
\label{XconnIFFK(X)conn}
Let $(X,\kappa)$ be a digital image. Then $X$ is $\kappa$-connected
if and only if $K(X,\kappa')$ is $\kappa'$-connected.
\end{thm}

\begin{proof}
Suppose $(X, \kappa)$ is connected. By
Proposition~\ref{DconnImpliesKDconn}, $K(X,\kappa')$ is
$\kappa'$-connected.

Conversely, suppose $K(X,\kappa')$ is $\kappa'$-connected.
By Proposition~\ref{disconnXimpliesdisconnKX}, $(X,\kappa)$ must
be connected.
\end{proof}

\begin{lem}
\label{pathToSingleton}
Let $(X,\kappa)$ be a digital image. Let $A$ be a finite
member of $K(X,\kappa')$. Then there is a path ${\cal P}$
in $K(A,\kappa')$ from a singleton to $A$.
\end{lem}

\begin{proof}
Let $x_0 \in A$. By Proposition~\ref{componentYieldsConnected},
there is a path in $K(A,\kappa')$ from $\{\,x_0\,\}$ to $A$.
\begin{comment}
Let $A = \{x_i\}_{i=0}^n$. Since $A$ is connected, there are
paths $P_i: [0,m_i]_{\Z} \to A$ from $x_{i-1}$ to $x_i$, $1 \le i \le n$.
Let $M = \sum_{i=1}^n m_i$.
Let $Q$ be the concatenation of paths 
\[ Q = P_1 \cdot P_2 \cdot \ldots \cdot P_n: [0, M]_{\Z} \to A.
\]
Then $Q(0) = \{x_0\}$, $Q([0, M]_{\Z}) = A$,  
and it is clear from Definition~\ref{hyperAdjDef} that
$Q([0,k-1]_{\Z}) \adjeq_{\kappa'} Q([0,k]_{\Z})$ for
$1 \le k \le M$. It follows that the function
${\cal P}: [0,M]_{\Z} \to K(A,\kappa')$ given by
${\cal P}(t) = Q([0,t]_{\Z})$ is a path in $K(A,\kappa')$ from
$\{x_0\}$ to $A$.
\end{comment}
\end{proof}

Suppose $(X,\kappa)$ is a connected digital image. We say
$Y \subset X$ {\em disconnects} $(X,\kappa)$ if $X \setminus Y$
is not $\kappa$-connected.

\begin{thm}
Let $(X,\kappa)$ be a connected digital image. 
Let $Y \subset X$. Let 
\begin{equation}
\label{calYdef}
    {\cal Y} = \{ B \in K(X,\kappa') ~|~ B \cap Y \neq \emptyset \}.
\end{equation} 
If $Y$ disconnects $(X,\kappa)$ then
${\cal Y}$ disconnects $K(X,\kappa')$.
\end{thm}

\begin{proof}
Suppose $Y$ disconnects $(X,\kappa)$. Then there are 
$x_0,x_1$ that are in distinct components of $X \setminus Y$.

Suppose ${\cal Y}$ fails to disconnect $K(X,\kappa')$. 
Then there exists a $\kappa'$-path
\begin{equation}
\label{pathInK(X)}
{\cal P} = \{B_j\}_{j=0}^n \subset K(X,\kappa') \setminus {\cal Y}
\end{equation}
from $\{x_0\}$ to $\{x_1\}$.
By Definition~\ref{hyperAdjDef}, there exist $y_j, z_j \in B_j$ such 
that $y_j \adjeq_{\kappa} z_{j+1}$ for $j < n$.
Since $B_j$ is connected, there are $\kappa$-paths
$P_0 \subset B_0$ from $x_0$ to $y_0$, $P_j \subset B_j$
from $z_j$ to $y_j$, and $P_n \subset B_n$ from $z_n$ to $x_1$.
Then $P = \bigcup_{j=0}^n P_j$ is a $\kappa$-path in 
$\bigcup_{j=0}^n B_j \subset X$ from $x_0$
to $x_1$. Since $Y$ disconnects $X$, we must have
$P \cap Y \neq \emptyset$. Hence for some $k$,
$B_k \cap Y \neq \emptyset$, contrary
to~(\ref{pathInK(X)}).
The contradiction establishes that
${\cal Y}$ disconnects $K(X,\kappa')$.
\end{proof}

\section{Multivalued functions and hyperspaces}
\label{multivaluedSec}
In this section, we examine relations between
various notions of continuous multivalued functions
between digital images, and hyperspaces of digital images.

\begin{definition}
\label{strongContDef}
A multivalued function 
$F: (X,\kappa) \multimap (Y,\lambda)$
\begin{itemize}
    \item has {\em strong continuity}
          {\rm \cite{Tsaur}} if for each pair of 
adjacent $x, y \in X$, every point of $F(x)$
is adjacent or equal to some point
of $F(y)$ and every point of $F(y)$ is adjacent or equal to
some point of $F(x)$;
    \item has {\em weak continuity} 
          {\rm \cite{Tsaur}} if for each pair of 
          adjacent $x, y \in X$, $F(x)$ and $F(y)$
          are adjacent sets in $Y$, i.e., there
          exist $a \in F(x)$, $b \in F(y)$ such that
          $a \adjeq_{\lambda} b$;
    \item is {\em connectivity preserving} 
          {\rm \cite{Kovalevsky}} if 
          $F(A) \subset Y$ is connected whenever 
          $A \subset X$ is connected;
    \item is {\em continuous}~\cite{EGS1,EGS2}
          if $X \subset \Z^n$,
          $\kappa = c_u$ for $1 \le u \le n$, and
          $F$ is {\em generated}
          by a continuous function $f: S(X,r) \to Y$
          for some positive integer $r$; where
          $S(X,r) = \bigcup_{x \in X} S(\{x\},r)$, where
          for $x = (x_1, \ldots, x_n)$, $S(\{x\},r)$
          is the set of all points
          $(y_1,\ldots,y_n)$ such that for each index~$i$
          we have $y_i = x_i + k_i/r$ for some
          integer $k_i$ such that $0 \le k_i < r$; 
          $S(X,r)$ inherits $c_u$ in the sense that
          $(y_1,\ldots,y_n) \adjeq_{c_u} (a_1,\ldots,a_n)$ 
          in $S(X,r)$ if for at most $u$ indices~$i$,
          $|y_i - a_i|=1/r$ and for all other indices~$j$,
          $y_j=a_j$; and ``$F$ is {\em generated}
          by $f$" means for all $x \in X$,
          $F(x) = \bigcup_{y \in S(\{x\},r)} \{f(y)\}$.
\end{itemize}
\end{definition}

We have the following.

\begin{thm}
\label{strongContInduces}
Let $F: (X,\kappa) \multimap (Y,\lambda)$ be a
strongly continuous multifunction between digital
images. Then the function 
$F_*: (2^X, \kappa') \to (2^Y, \lambda')$ 
defined by $F_*(A) = F(A)$ is continuous.
\end{thm}

\begin{proof}
Let $A_0 \adj_{\kappa'} A_1$ in $2^X$. We must show
that $F(A_0) \adjeq_{\lambda'} F(A_1)$ in $2^Y$.

Let $x \in A_0$, $y \in A_1$ such that
$x \adjeq_{\kappa} y$. By Definition~\ref{strongContDef},
for every $p \in F(x)$ there exists $q \in F(y)$
such that $p \adjeq_{\lambda} q$. Similarly, given
$u \in A_1$, $v \in A_0$ such that $u \adj_{\kappa} v$,
for every $r \in F(u)$ there exists $s \in F(v)$ such
that $r \adjeq_{\lambda} s$. The assertion follows.
\end{proof}

The following shows that in substituting
weak continuity, continuity, or connectivity-preserving
for strong continuity, we fail to
obtain a result analogous to
Theorem~\ref{strongContInduces}.

\begin{exl}
Let $F: ([0,1]_{\Z},c_1) \multimap ([0,2]_{\Z},c_1)$
be defined by $F(0)=\{0\}$, $F(1) = \{1,2\}$. Then $F$ has
weak $c_1$-continuity, is $c_1$-continuous, and is
$c_1$-connectivity-preserving, but since $2 \in F(1)$
has no $c_1$-neighbor in $F(0)$, the induced function
$F_*: (2^X, c_1') \to (2^Y, c_1')$ is not
$(c_1',c_1')$-continuous.
\end{exl}

\section{Cycles and Girth}
\label{cycleSec}
The reader is reminded that:
\begin{itemize}
    \item a point in $2^X$ is a nonempty subset of $X$;
    \item a cycle in $X$ is a closed path of at least 3 distinct
          points in which no node repeats, but in which a point $x$ 
          can be adjacent to points distinct from the 
          predecessor and successor of $x$ in the path (the 
          cycle does not need to be chordless).
\end{itemize}

\begin{prop}
\label{cyclesIFF}
Let $(X,\kappa)$ be a digital image. Then $K(X,\kappa')$
has a 3-cycle if and only if $(X,\kappa)$ has a non-isolated point.
\end{prop}

\begin{proof}
It is elementary that if the points of $(X,\kappa)$ are all isolated,
then $K(X,\kappa')$ has no cycle.

Suppose $x \in X$ is not isolated in $(X,\kappa)$. Then there exists
$y \in X$ such that $x \adj_{\kappa} y$. Then
$\left \{ \{x\}, \{x,y\}, \{y\} \right \}$ is a 3-cycle in
$K(X,\kappa')$.
\end{proof}

The {\em girth} of a graph $(X,\kappa)$ is variously described in
the literature as the length of a shortest or of
a longest~\cite{Berge} cycle in $(X,\kappa)$. We may 
distinguish these concepts as {\em girth} and {\em Girth}, respectively. In light of Proposition~\ref{cyclesIFF}, the 
Girth is more interesting, so in the following we focus on Girth.

\begin{exl}
\label{6cycle}
If $(X,\kappa)$ is a digital image and $x \in X$ such that
$N(X,x, \kappa)$ has distinct points $u$ and $v$ that are
not $\kappa$-adjacent, then
$K(X,\kappa')$ has Girth of at least 6.
\end{exl}

\begin{proof}
By hypothesis, there exist distinct $u,v \in N(X,x,\kappa)$.
Then by Definition~\ref{hyperAdjDef}, $K(X,\kappa')$ has a 6-cycle
\[ \{u\}, \{u,x\}, \{u,x,v\}, \{x,v\}, \{v\}, \{x\}.
\]
\end{proof}

\begin{exl}
The Girth of $(2^{[1,4]_{\Z}}, c_1')$ is 15, which is equal
to $\#(2^{[1,4]_{\Z}}, c_1')$. 
I.e., $(2^{[1,4]_{\Z}}, c_1')$ has
a cycle containing all members of $(2^{[1,4]_{\Z}}, c_1')$.
\end{exl}

\begin{proof}
It is easy to see that the following sequence of the 15
distinct members of $(2^{[1,4]_{\Z}}, c_1')$ is a
$c_1'$-cycle.
\[ \{1,2\}, \{1,2,3\}, \{1,3\}, \{1,4\}, \{1,3,4\}, 
   \{1,2,4\}, \{1,2,3,4\}, \{2,3,4\}, \{2,3\}, 
\]
\[ \{2,4\}, \{3,4\}, \{4\}, \{3\}, \{2\}, \{1\}
\]
\end{proof}

\section{Dominating set}
A subset $D$ of a graph $(X,\kappa)$ is a 
{\em dominating set} for, or {\em dominates},
$(X,\kappa)$, if given $x \in X$ there exists
$d \in D$ such that $d \adjeq_{\kappa} x$.

\label{dominateSec}
\begin{thm}
Let $(X,\kappa)$ be a digital image and let $D \subset X$. Let
\[ {\cal D} = \{A \in 2^X ~|~ A \cap D \neq \emptyset\}
\]
Then $D$ dominates $(X,\kappa)$ if and only if
${\cal D}$ dominates $(2^X, \kappa')$.
\end{thm}

\begin{proof}
Suppose $D$ dominates $(X,\kappa)$. Let $x \in A \in 2^X$.
There exists $y \in D$ such that $x \adjeq_{\kappa} y$.
It follows from Definition~\ref{hyperAdjDef} that
\[ A' = A \cup \{y\} \adjeq_{\kappa'} A.
\]
Since $A$ is arbitrary and $A' \in {\cal D}$,
it follows that ${\cal D}$ dominates
$(2^X, \kappa')$.

Suppose ${\cal D}$ dominates $(2^X, \kappa')$.
Let $x \in X$. Then there exists 
$A \in {\cal D}$ such
that $A \adjeq_{\kappa'} \{x\}$. Therefore, for all $a \in A$
we have $a \adjeq_{\kappa} x$. Since there exists
$d \in A \cap D$, $d \adjeq_{\kappa} x$. Thus, $D$
dominates $X$.
\end{proof}

\begin{comment}
\begin{thm}
Let $(X,\kappa)$ and $(Y,\lambda)$ be digital images.
Let $f: X \to Y$ be $(\kappa,\lambda)$-continuous.
Let $D$ dominate $(X,\kappa)$. Then $f(D)$ dominates $f(X)$.
\end{thm}

\begin{proof}
Let $y \in f(X)$ and let $x \in f^{-1}(y)$. There
exists $d \in D$ such that $x \adj_{\kappa} d$. Then
$y = f(x) \adjeq_{\lambda} f(d)$. The assertion follows.
\end{proof}
\end{comment}

\section{Diameter}
\label{diamSec}
\begin{definition}
{\rm \cite{Han05}}
Let $(X,\kappa)$ be a connected graph. The 
{\em shortest path metric} for $(X,\kappa)$ is
\[ d_{\ell}(x,y) = \min\{length(P) ~|~ P 
   \mbox{ is a $\kappa$-path in $X$ from $x$ to $y$}\},
   \mbox{ for } x,y \in X.
\]
\end{definition}

\begin{definition}
The {\em diameter} of a finite connected graph $(X,\kappa)$
is
\[ diam(X,\kappa) = \max\{d_{\ell}(x,y) ~|~ x,y \in X\}.
\]
\end{definition}

\begin{comment}
Commented out due to error in "proof"

\begin{thm}
Let $X$ be a nonempty finite $\kappa$-connected
digital image. Let $n = diam(X,\kappa)$. Then 
$diam(K(X,\kappa')) \ge n$.
\end{thm}

\begin{proof}
Let $x,y \in X$ be such that 
\[ d_{\ell}(x,y) = diam(X,\kappa) = n.
\]
Then we claim a shortest path in $K(X,\kappa')$
from $\{x\}$ to $\{y\}$ has length $n$.

To prove this claim, we argue as follows. By choice of
$x$ and $y$, there exists $\{x_i\}_{i=0}^n \subset X$
such that $x_0=x$, $x_n=y$, and $x_i \adj_{\kappa} x_{i+1}$
for $0 \le i < n$. Then we see easily that
$\{\{x_i\}\}_{i=0}^n$ is a path of length $n$
in $K(X,\kappa')$ from $\{x\}$ to $\{y\}$.

Suppose there is a path $\{A_j\}_{j=0}^m$
in $K(X,\kappa')$  from $\{x\}=A_0$ to $A_m=\{y\}$ such that
$m < n$. Then by Definition~\ref{hyperAdjDef}, there
exists $a_1 \in A_1$ such that $x = a_0 \adjeq_{\kappa} a_1$,
and by induction $a_{j+1} \in A_{j+1}$ such that
$a_j \adjeq_{\kappa} a_{j+1}$ for $1 \le j < m$, where we must
have $a_m = y$. But then $\{a_j\}_{j=0}^m$ is a path of
length $m < n$ from $x$ to $y$ in $(X,\kappa)$, contrary to
our choices of $n$, $x$, and $y$. 
The claim, and the assertion, follow.
\end{proof}
\end{comment}

\begin{definition}
{\rm \cite{Berge}}
Let $(X,\kappa)$ be a connected digital image. 
For $x \in X$, the
{\em associated number $e(x)$ of $x$} is
\[ e(x) = \max \{d_{\ell}(x,y) ~|~ y \in X\}.
\]
A {\em center} of $(X,\kappa)$
is a point $x_0 \in X$ such that
\[ e(x_0) = \min \{e(x) ~|~ x \in X\}.
\]
The associated number of the center is the
{\em radius} of $(X,\kappa)$.
\end{definition}

We have the following.

\begin{thm}
Let $(X,\kappa)$ be a finite connected digital image with
radius $r$. Let $\#X = n$. Then
$diam(K(X,\kappa')) < 2(n + r - 1)$.
\end{thm}

\begin{proof}
Let $x_0$ be a center of $(X,\kappa)$. Let
$A_0,A_1 \in K(X,\kappa')$. Let $y_0 \in A_0$,
$y_1 \in A_1$. By assumption, there are
paths $P_i$ of length at most $r$ from $x_0$ to
$y_i$. Thus, $P_0 \cup P_1$ is a $\kappa$-path 
in $X$ of length at most
$2r$ from $y_0$ to $y_1$. It follows from
Definition~\ref{hyperAdjDef} that
${\cal P} = \{\{p\} ~|~ p \in P_0 \cup P_1\}$ is a 
$\kappa'$-path in $K(X)$ of length at most
$2r$ from $\{y_0\}$ to $\{y_1\}$.

Let $Q_0 = \{y_0\}$. We argue inductively as follows.
Suppose we have $Q_k \in K(X,\kappa')$ such that
$Q_k \subset A_0$. If $Q_k \neq A_0$,
then since $A_0$ is connected, there exists
$q \in A_0 \setminus Q_k$ such that for some
$q' \in Q_k$, $q \adj_{\kappa} q'$. By
Definition~\ref{hyperAdjDef}, we have
\[ Q_{k+1} = Q_k \cup \{q'\} \adj_{\kappa'}
          Q_k.
\]
Since $Q_{\#A_0 - 1} = A_0$, the set
${\cal P}_0 = \{Q_j\}_{j=0}^{\#A_0-1}$
is a path in $K(X,\kappa')$ of length $\#A_0 - 1$ from 
$\{y_0\}$ to $A_0$; equivalently, from $A_0$ to $\{y_0\}$.

Similarly, we can construct a path ${\cal P}_1$
in $K(X,\kappa')$ of length $\#A_1 - 1$ from 
$\{y_1\}$ to $A_1$. Therefore,
${\cal P}_0 \cup {\cal P} \cup {\cal P}_1$
is a path in $K(X,\kappa')$ of length at most
\[ \#A_0 - 1 + 2r + \#A_1 - 1 \le 2(n+r-1)
\]
from $A_0$ to $A_1$.
Further, we may assume $\min \{\#A_0, \#A_1\} < n$;
since otherwise $A_0 = X = A_1$, so there is a path
of length 0 from $A_0$ to $A_1$ in $K(X,\kappa')$. It
follows that for any $A_0, A_1 \in K(X,\kappa')$ there is
a path in $K(X,\kappa')$ from $A_0$ to $A_1$ of length
less than $2(n+r-1)$. The assertion follows.
\end{proof}

\begin{comment}
\begin{thm}
Let $(X,\kappa)$ and $(Y,\lambda)$ be finite
connected digital images. Let
$f \adj_{\Phi(\kappa,\lambda)} g$ in $Y^X$. Let $d$ be a metric
for $Y$ and let $M \in \R$ be such that
$a \adj_{\kappa} b$ in $Y$ implies $d(a,b) \le M$. Then
$|diam_d(f(X),\lambda) - diam_d(g(X),\lambda)| \le 2M$.
\end{thm}

\begin{proof}
Let $y_0, y_1 \in f(X)$ be such that
$d(y_0,y_1) = diam_d(f(X))$. Let
$x_0 \in f^{-1}(y_0)$, $x_1 \in f^{-1}(y_1)$. Then
\[ y_i' = g(x_i) \adjeq_{\lambda} f(x_i) = y_i \mbox{ for } i \in \{0,1\}. 
\]
Therefore, 
\begin{equation}
\label{diamCompare1}
    diam_d(f(X)) = d(y_0,y_1) \le d(y_0, y_0') + d(y_0', y_1') +
   d(y_1',y_1) \le diam_d(g(X)) + 2M.
\end{equation} 
Similarly, 
\begin{equation}
\label{diamCompare2}
    diam_d(g(X)) \le diam_d(f(X)) + 2M.
\end{equation}
The assertion follows from~(\ref{diamCompare1})
and~(\ref{diamCompare2}).
\end{proof}
\end{comment}

\section{Further remarks}
\label{concludeSec}
We have introduced into digital topology the study of hyperspaces
of digital images, and have taken a somewhat different 
approach to function graphs than that introduced
in~\cite{LuptonEtal}. We have studied some
relations between digital hyperspaces and
digital function graphs. We have examined a number of
properties of digital hyperspaces concerning cardinality, continuous maps and homotopy, connectivity,
cycles and Girth, dominating sets, and diameters.

Suggestions from anonymous reviewers are gratefully acknowledged.
%\end{acknowledgements}

\section{Declarations}
This research was not supported by a grant, nor by any
organizational funding. The author has no relevant financial
or non-financial interests to disclose.

%\end{article}
\end{document}